\documentclass[preprint,12pt]{elsarticle}

\usepackage{amsfonts}
\usepackage{amssymb}
\usepackage{amsthm}
\usepackage{amsmath}
\usepackage{tikz, subfigure}
\usepackage{lipsum}

\newtheorem{Theorem}{Theorem}[section]

\newtheorem{Lemma}{Lemma}[section]
\newtheorem{Corollary}{Corollary}[section]

\newtheorem{Definition}{Definition}[section]

\journal{Indagationes Mathematicae}

\begin{document}

\begin{frontmatter}



\title{Shrinking random $\beta$-transformation}


\author{Karma Dajani and Kan Jiang}

\address{Department of Mathematics, Utrecht University, Fac Wiskunde en informatica and MRI, Budapestlaan 6, P.O. Box 80.000, 3508 TA Utrecht, The Netherlands}

\begin{abstract}
 For any $n\geq 3$, let $1<\beta<2$ be the largest positive real number satisfying the equation  $$\beta^n=\beta^{n-2}+\beta^{n-3}+\cdots+\beta+1.$$ In this paper we define the shrinking random $\beta$-transformation $K$ and investigate natural invariant measures for $K$, and the induced transformation of $K$ on a special subset of the domain. We prove that both transformations have a unique measure of maximal entropy.  However, the measure induced from the intrinsically ergodic measure for $K$ is not the intrinsically ergodic measure for the induced system.
\end{abstract}


\end{frontmatter}


\section{Introduction}

Let $\beta\in(1, 2)$ and $x\in \mathcal{A}_{\beta}=[0, (\beta-1)^{-1}]$, we call a sequence $(a_n)_{n=1}^{\infty}\in\{0,1\}^{\mathbb{N}}$ a $\beta$-expansion of $x$ if $$x=\sum_{n=1}^{\infty}\dfrac{a_n}{\beta^n}.$$  Renyi \cite{Renyi} introduced the greedy map, and showed that  the greedy expansion $(a_i)_{i=1}^{\infty}$ of
$x\in [0,1)$ can be generated by defining $T(x)=\beta x \mod 1$ and letting $a_i=k$
whenever  $T^{i-1}(x)\in [k\beta^{-1},(k+1)\beta^{-1})$.
Since then, many papers were dedicated to the dynamical properties of this map, see for example \cite{Walters, Kalle, Tom1, Par64, DajaniDeVrie, Parry} and references therein.
However, Renyi's greedy map is not the unique dynamical approach to generate $\beta$-expansions. In \cite{KarmaCor}  (see also \cite{KM, DajaniDeVrie}) a new transformation was introduced, the random
$\beta$-transformation, that generates all possible $\beta$-expansions, see Figure 1. This transformation makes random choices between the maps $T_0(x)=\beta x$ and $T_1(x)=\beta x-1$ whenever the orbit falls into 
$[\beta^{-1},\beta^{-1}(\beta-1)^{-1}]$, which we refer to as the switch region.

 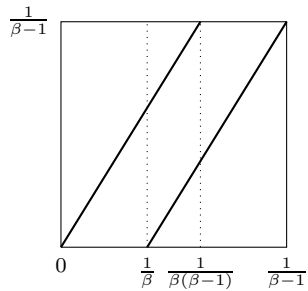
\begin{figure}[h]\label{figure1}
\centering
\begin{tikzpicture}[scale=3]
\draw(0,0)node[below]{\scriptsize 0}--(.382,0)node[below]{\scriptsize$\frac{1}{\beta}$}--(.618,0)node[below]{\scriptsize$\frac{1}{\beta(\beta-1)}$}--(1,0)node[below]{\scriptsize$\frac{1}{\beta-1}$}--(1,1)--(0,1)node[left]{\scriptsize$\frac{1}{\beta-1}$}--(0,.5)--(0,0);
\draw[dotted](.382,0)--(.382,1)(0.618,0)--(0.618,1);
\draw[thick](0,0)--(0.618,1)(.382,0)--(1,1);
\end{tikzpicture}\caption{The dynamical system for $\{T_{0}(x)=\beta x,\,T_{1}(x)=\beta x-1\}$}
\end{figure}

Although, all possible  $\beta$-expansions can be generated via the random $\beta$-transformation, nevertheless, for some practical problems one would want to make choices only on a subset of the switch region $[\beta^{-1},\beta^{-1}(\beta-1)^{-1}]$, for instance, 
in A/D (analog-to-digit) conversion \cite{Wang}. This motivates our study of 
the shrinking random $\beta$-transformation described below.

Let $1<\beta<2^{-1}(1+\sqrt{5})$,
 $\Omega=\{0,1\}^{\mathbb{N}}$, and $E=[0,(\beta-1)^{-1}]$. Set $a=(\beta^2-1)^{-1}, b=\beta(\beta^2-1)^{-1}$, i.e. $T_0(a)=b, T_1(b)=a$. 
The shrinking random $\beta$-transformation $K$ is defined in the following way.
 \begin{Definition}
 $K:\Omega\times E\to \Omega\times E$ is defined by
\begin{equation*}
K(\omega, x)=\left\lbrace\begin{array}{cc}
                (\omega, \beta x)& x\in[0,a)\\
       (\sigma(\omega), \beta x-\omega_1)& x\in[a,b]\\
              (\omega, \beta x-1)& x\in(b,(\beta-1)^{-1}]\\
                \end{array}\right.
\end{equation*}
\end{Definition}

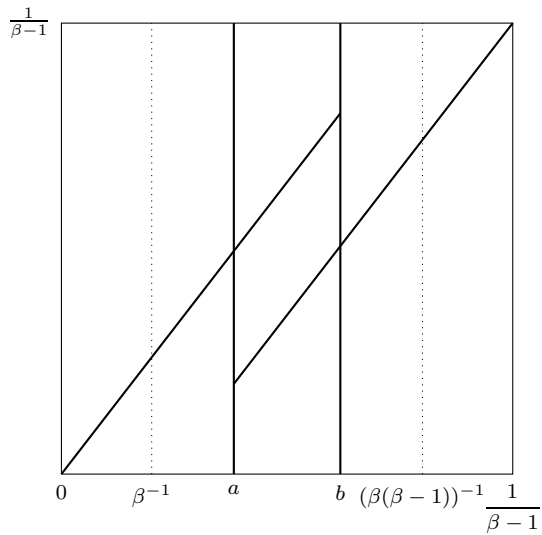
\begin{figure}[h]\label{figure1}
\centering
\begin{tikzpicture}[scale=6]
\draw(0,0)node[below]{\scriptsize 0}--(.2,0)node[below]{\scriptsize$\beta^{-1}$}--(.382,0)node[below]{\scriptsize$a$}--(.618,0)node[below]{\scriptsize$b$}--(.8,0)node[below]{\scriptsize$(\beta(\beta-1))^{-1}$}--(1,0)node[below]{\scriptsize$\dfrac{1}{\beta-1}$}--(1,1)--(0,1)node[left]{\scriptsize$\frac{1}{\beta-1}$}--(0,.5)--(0,0);
\draw[thick](.382,0)--(.382,1)(0.618,0)--(0.618,1);
\draw[thick](0,0)--(0.618,0.8)(.382,0.2)--(1,1);
\draw[dotted](.2,0)--(.2,1)(0.8,0)--(0.8,1);
\end{tikzpicture}\caption{Shrinking random $\beta$-transformation}
\end{figure}

\noindent Given  $(\omega,x)\in \Omega\times[a,b]$, the first return time is defined by 
$$\tau(\omega, x)=\min\{n\geq 1:  K^{i}(\omega,x)\notin \Omega\times[a,b], 1\leq  i\leq n-1, K^{n}(\omega,x)\in \Omega\times[a,b]\},$$ 
Define $K_{\Omega\times[a,b]}(\omega,x)=K^{\tau(\omega,x)}( \omega,x)$, and denote it for simplicity by $I$. 

We now consider a special family of algebraic bases defined as follows.  For any $n\geq 3$, let $1<\beta<2$ be the largest positive real number satisfying the equation  $$\beta^n=\beta^{n-2}+\beta^{n-3}+\cdots+\beta+1.$$  The following lemma is clear. 
\setcounter{Lemma}{1}
\begin{Lemma}\label{Goldmean}
For any $n\geq 3$, let $\beta_n>1$ be the largest positive real root of the equation 
\begin{equation}\label{Pisot}
x^n=x^{n-2}+x^{n-3}+\cdots+x+1.
\end{equation}
Then $(\beta_n)$ is an increasing sequence which converges to $2^{-1}(1+\sqrt{5})$.
\end{Lemma}

Throughout the paper we will assume $\beta=\beta_n$ for some $n\ge 3$.  In Section 2, we will show that $I$ can be identified with a full left shift. As a result it will be easy to find $I$-invariant measures, and to show that $I$ is intrinsically ergodic (i.e. has a unique measure of maximal entropy).  In the last section, we identify the dynamics of $K$ with a topological Markov chain, and then use Parry's recipe to prove the following result.
\setcounter{Theorem}{2}
 \begin{Theorem}\label{Main11111}
 For any $n\geq 3$, let $1<\beta<2$ be the largest positive real number satisfying the equation  $$\beta^n=\beta^{n-2}+\beta^{n-3}+\cdots+\beta+1.$$ Then  the  shrinking random $\beta$-transformation $K$ and the induced transformation  $I=K_{\Omega\times[a,b]}$ have intrinsically ergodic measures.  Moreover, the measure induced from the intrinsically ergodic measure of $K$ on $\Omega\times [a,b]$, does not yield the unique measure of  maximal entropy for $I$. 
\end{Theorem}

\section{Invariant measures for $I=K_{\Omega\times[a,b]}$}

As above, $\beta$ satisfies $\beta^n=\beta^{n-2}+\beta^{n-3}+\cdots+\beta+1$, $n\ge 3$.
It is easy to check that for $(\omega ,x)\in \Omega\times [a,b]$,  the first return time $\tau(\omega,x)\in \{2,3,\cdots,n\}$. We give a simple proof for this statement. Since $a=T_1(b), b=T_0(a)$ and the fact that $\beta$ satisfies the equation $\beta^n=\beta^{n-2}+\beta^{n-3}+\cdots+\beta+1$,  it follows that the largest return time of $a$ is $n$. Similarly, we can show that those points,  which are very close to $a$ or $b$, have the return time  $2$.  What we want to emphasize here is that we delete these points with return time $1$. As there are only countable  such points, we can delete these points without affecting our result.

Consider the space $ \Omega\times \{2,3,\cdots,n\}^{\mathbb{N}}$ equipped with the product $\sigma$-algebra, and the left shift $\sigma^{\prime}$. Define the map
$\phi: (\Omega\times [a,b])\setminus(\cup_{i=0}^{\infty} K^{-i}(\Omega\times\{a\} \cup \Omega\times\{b\}))\to \Omega\times \{2,3,\cdots,n\}^{\mathbb{N}}$ by
$$\phi(\omega,x)=(\omega,(n_1,n_2,\cdots,n_k,\cdots)),$$
where $n_i$ is the $i$-th return time of $(\omega,x)$ to $(\Omega\times [a,b])\setminus(\cup_{i=0}^{\infty} K^{-i}(\Omega\times\{a\} \cup \Omega\times\{b\})$, i.e.
$n_i=\tau(I^{i-1}(\omega,x))$.

Given $(a_n)\in\{0,1\}^{\mathbb{N}}$, we denote the value of the sequence $(a_n)$ by $(a_n)_{\beta}=\sum_{n=1}^{\infty}a_n\beta^{-n}$.
\begin{Lemma}\label{expansions}
The sequences $$(01)^{j_1}(1\underbrace{0\cdots0}_{n-1})^{j_2}(01)^{j_3}(1\underbrace{0\cdots0}_{n-1})^{j_4}\cdots $$
and 
$$(10)^{j_1}(0\underbrace{1\cdots1}_{n-1})^{j_2}(10)^{j_3}(0\underbrace{1\cdots1}_{n-1})^{j_4}\cdots $$
 are the possible $\beta$-expansions of $a$ and $b$ generated by the map $K$ respectively, where $0\leq j_k\leq \infty$.
\end{Lemma}

\begin{proof}
The proof follows from the fact that $a=T_1(b)=T_0^{n-1}T_1(a)$ and $b=T_0(a)=T_1^{n-1}T_0(b)$. 
\end{proof}

\begin{Lemma}\label{bijection}
$\phi$ is a measurable bijection.
\end{Lemma}
\begin{proof}
Firstly we prove $\phi$ is one-to-one. Let $\phi(\omega,x_1)=\phi(\tau,x_2)$. Then we have $\omega=\tau$, and the first return time functions coincide. We denote the values of  this function   by $(n_i)^{\infty}_{i=1}$. Since $x_1\in [a, b]$ and $\omega=(\omega_1\omega_2\omega_3\cdots)$, by the definition of $K$, we choose the first digit of $x_1$ by $\omega_1$. Then the orbit of $x_1$ jumps out of $[a,b]$, and in  the region $[0, a)\cup (b, (\beta-1)^{-1}]$, we can only choose digits $\overline{\omega_1}$ for $n_1-1$ times.  After $n_1$ times later, the orbit of $x_1$ goes back to $[a,b]$. Therefore we can implement similar algorithm again. Using this idea, 
one can easily check that 
$$x_1=(\omega_1\overline{\omega_1}^{n_1-1}\omega_2\overline{\omega_2}^{n_2-1}\cdots)_{\beta},$$
and 
$$x_2=(\tau_1\overline{\tau_1}^{n_1-1}\tau_2\overline{\tau_2}^{n_2-1}\cdots)_{\beta}$$
where $\overline{\omega_i}=1-\omega_i$, and $ (\omega_i)^{k}$ means $k$ consecutive $\omega_i$.

As such we have $x_1=x_2$. 
Now we prove $\phi$ is also a surjection. Given any $(\omega, (n_1,n_2,\cdots,n_k,\cdots))$,  
it is sufficient to show that $$x=(\omega_1\overline{\omega_1}^{n_1-1}\omega_2\overline{\omega_2}^{n_2-1}\cdots)_{\beta}\in [a,b]$$

We decompose the sequence $(\omega_1\overline{\omega_1}^{n_1-1}\omega_2\overline{\omega_2}^{n_2-1}\cdots)$  into  the blocks $\omega_i\overline{\omega_i}^{n_i-1}$. Note that for any $i\geq 1$, the value of the block can be classified in the following way:
 
 \noindent If $\omega_i=0$ and $n_i=2$, then
$$(\omega_i\overline{\omega_i}^{n_i-1})_{\beta}= (01)_{\beta}.$$
 If $\omega_i=0$ and $n_i\geq 3,$ then 
$$(\omega_i\overline{\omega_i}^{n_i-1})_{\beta}\geq (1\underbrace{0\cdots0}_{n-1})_{\beta}.$$
 If $\omega_i=1$, then 
$$(\omega_i\overline{\omega_i}^{n_i-1})_{\beta}\geq (1\underbrace{0\cdots0}_{n-1})_{\beta}.$$
Here we use the fact  $1<\beta<\dfrac{\sqrt{5}+1}{2}$, see Lemma \ref{Goldmean}.
Hence, we have  $$x=(\omega_1\overline{\omega_1}^{n_1-1}\omega_2\overline{\omega_2}^{n_2-1}\cdots)_{\beta}\geq ((01)^{j_1}(1\underbrace{0\cdots0}_{n-1})^{j_2}(01)^{j_3}(1\underbrace{0\cdots0}_{n-1})^{j_4}\cdots )_{\beta}=a.$$
or $$x=(\omega_1\overline{\omega_1}^{n_1-1}\omega_2\overline{\omega_2}^{n_2-1}\cdots)_{\beta}\geq ((1\underbrace{0\cdots0}_{n-1})^{j_1}(01)^{j_2}(1\underbrace{0\cdots0}_{n-1})^{j_3}(01)^{j_4}\cdots )_{\beta}=a.$$
Similarly, we  prove by symmetry that 
 $$\bar{x}=(\beta-1)^{-1}-x=(\overline{\omega_1}\omega_1^{n_1-1}\overline{\omega_2}\omega_2^{n_2-1}\cdots)_{\beta}\geq  ((01)^{j_1}(1\underbrace{0\cdots0}_{n-1})^{j_2}(01)^{j_3}(1\underbrace{0\cdots0}_{n-1})^{j_4}\cdots )_{\beta}=a$$
 or
 $$\bar{x}=(\beta-1)^{-1}-x=(\overline{\omega_1}\omega_1^{n_1-1}\overline{\omega_2}\omega_2^{n_2-1}\cdots)_{\beta}\geq ((1\underbrace{0\cdots0}_{n-1})^{j_1}(01)^{j_2}(1\underbrace{0\cdots0}_{n-1})^{j_3}(01)^{j_4}\cdots )_{\beta}=a.$$
Since $b=(\beta-1)^{-1}-a$, we have $a\leq x\leq b$ and $\phi$ is surjective. It remains to show that $\phi$ is measurable. For any cylinders $C=\{\omega\in\Omega:\omega_1=i_1,\cdots, \omega_m=i_m\}$ and $D=\{y\in  \{2,3,\cdots,n\}^{\mathbb{N}} :y_1=n_1,\cdots, y_m=n_m\}$, we have
$$\phi^{-1}(C\times D)=\{(\omega,x)\in \Omega\times [a,b]: \tau(\omega,x)=n_1, \tau(I(\omega,x))=n_2\cdots,\tau(I^{m-1}(\omega,x))=n_m\}$$
which is a measurable set, since $\tau$ and $I$ are measurable.
\end{proof}

 \begin{Lemma}\label{isomorphism}
 Let $\mu$ be any $\sigma\times \sigma^{\prime}$-invariant measure on $\Omega\times \{2,3,\cdots,n\}^{\mathbb{N}}$. Then, the measure $\mu\circ \phi$ is $I$-invariant, and the dynamical systems $(\Omega\times [a,b], I, \mu\circ \phi)$, and $(\Omega\times \{2,3,\cdots,n\}^{\mathbb{N}}, \sigma\times\sigma^{'}, \mu)$ are isomorphic.
\end{Lemma}

\begin{proof} It is easy to check that $(\sigma\times\sigma^{'}) \circ \phi= \phi\circ I$.
 Since $\phi$ is a measurable bijection, $\mu\circ \phi$ is $I$-invariant and the result follows. 
\end{proof}
\setcounter{Corollary}{3}
\begin{Corollary} Let $m_p$ be the $(p,1-p)$ product measure on $\Omega$, and $\mu_{\pi}$ the product measure on $\{2,3,\cdots,n\}^{\mathbb{N}}$ induced by the probability vector $\pi=(\pi_2,\cdots,\pi_n)$, i.e. $\mu_{\pi}(\{(a_n)\in \{2,3,\cdots,n\}^{\mathbb{N}}:a_1=i_i, \cdots, a_m=i_m\})=\pi_{i_1}\cdots \pi_{i_m}$. Then, $(m_p\times\mu_{\pi})\circ \phi$ is an $I$-invariant ergodic measure on $\Omega\times [a,b]$.
\end{Corollary}

\begin{proof} Note that $\mu_{\pi}$ is $\sigma^{\prime}$-invariant, and since $\sigma$ is weakly mixing, we have that $(m_p\times\mu_{\pi})$ is $\sigma\times \sigma^{\prime}$-invariant ergodic measure. By Lemma \ref{isomorphism}, it follows that $(m_p\times\mu_{\pi})\circ \phi$ is an $I$-invariant ergodic measure on $\Omega\times [a,b]$.

\end{proof}
Note that for different probability vectors $\pi^{(1)}$ and $\pi^{(2)}$, the corresponding measures $(m_p\times\mu_{\pi^{(1)}})\circ \phi$ and $(m_p\times\mu_{\pi^{(2)}})\circ \phi$ are singular with respect to each other. It is natural to ask the following question:
 when do  we have $(m_p\times\mu_{\pi})\circ \phi=m_p\times \lambda$, where $\lambda$ is the  normalized Lebesgue measure on $[a,b]$?

 To answer this question, we need to find an explicit expression for the induced transformation $K|_{\Omega\times[a,b]}$ in terms of the first return time. We begin by partitioning $[a,b]$ using the greedy orbits, i.e. when $x\in[a,b]$ we implement $T_1$ on $x$. Define the greedy map $L_1(x)=\beta^{n-i+1}x-\beta^{n-i}, $ where $x\in [c_i, c_{i+1})$, $c_1=a, c_n=b$,
$c_i=\beta^{i-1}a-\beta^{i-2}+\beta^{-1},\,2\leq  i\leq n-1$. 
Similarly, we can define the lazy map $L_0(x)$ (we choose $T_0$ if the orbits fall into $[a,b]$) by 
$$L_0(x)= \beta^{n-i+1}x-\beta^{n-i-1}-\beta^{n-i-2}-\beta^{n-i-3}-\cdots-\beta-1,$$
if $x\in (d_i,d_{i+1}]$, where  $d_1=a, d_n=b$,
$d_i=\beta^{n-i}b-\beta^{n-i-2}-\cdots \beta-1-\beta^{-1},\,2\leq  i\leq n-1$. 
It is easy to see that  $L_1$ and $ L_0$  are Generalized L\"uroth Series (GLS) maps \cite{KarmaCor1}.
Hence, the induced transformation $I=K|_{\Omega\times[a,b]}$ is given by $I(\omega,x)=(\sigma(\omega), L_{\omega_1}(x))$. 
We now answer the question posed above.
\setcounter{Theorem}{4}
\begin{Theorem}
Let $P=\left(\dfrac{p_1}{b-a}, \dfrac{p_2}{b-a},\cdots, \dfrac{p_{n-1}}{b-a}\right)$, where $p_i=\beta^{i}a-\beta^{i-1}a-(\beta^{i-1}-\beta^{i-2}), 1\leq i\leq n-2$, $p_{n-1}=b-\beta^{n-2}a-\beta^{n-3}a+\beta^{-1}$.  
Then, $m_p\times \lambda$ is an $I$-invariant ergodic measure and  
$(m_p\times P)\circ \phi=m_p\times \lambda$. 
\end{Theorem}
\begin{proof}
By \cite[Theorems 1]{KarmaCor1}, the GLS maps $L_0$ and $L_1$ preserve  the normalized Lebesgue measure. Since the induced transformation $I$ is a skew product, it follows that $m_p\times \lambda$ is an $I$-invariant measure.  To show $(m_p\times P)\circ \phi=m_p\times \lambda$, it is enough to show that
$(m_p\times P)=(m_p\times \lambda)\circ \phi^{-1}$. Let $C=\{\omega\in\Omega:\omega_1=i_1,\cdots, \omega_m=i_m\}$ and $D=\{y\in  \{2,3,\cdots,n\}^{\mathbb{N}} :y_1=n_1,\cdots, y_m=n_m\}$. Then,
$$\phi^{-1}(C\times D)=\{(\omega,x)\in \Omega\times [a,b]: \tau(\omega,x)=n_1, \cdots,\tau(I^{m-1}(\omega,x))=n_m\}=C\times J,$$
where 
$$J=D_{n_1}\cap L_{i_1}^{-1}(D_{n_2})\cap (L_{i_2}\circ L_{i_1})^{-1}(D_{n_3})\cap \cdots \cap (L_{i_{m-1}}\circ \cdots  \circ L_{i_{1}})^{-1}(D_{n_m})$$
with $D_{n_j}=[c_{n-n_j+1}, c_{n-n_j+2})$ if $i_j=1$ and $D_{n_j}=(d_{n_j-1},d_{n_j}]$ if $i_j=0$.
Since the maps $L_0$ and $L_1$ are piecewise linear and surjective, an easy calculation shows that $J$ is an interval of length $\dfrac{p_{n_1}\cdots p_{n_m}}{(b-a)^m}=P(D)$, see \cite[Theorem 1]{KarmaCor1}. Thus,
$$(m_p\times \lambda)\bigg(\phi^{-1}(C\times D)\bigg)=m_p(C)P(D)=(m_p\times P)(C\times D).$$
\end{proof}

Now, we turn our attention in finding the intrinsically ergodic measure for $I=K_{\Omega\times[a,b]}$, i.e. the unique measure of maximal entropy. For this, we will identify the dynamics of $I$ with a full left shift. Consider the space $$\Lambda=\{(0,2), (0,3),\cdots, (0,n),(1,2), (1,3),\cdots, (1,n)\}^{\mathbb{N}}$$
Here the first coordinate denotes the outcome of the coin toss (heads=0 or tails=1), and the second denotes the return time to $\Omega\times [a,b]$. Let $S$ be the left shift on $\Lambda$, i.e. $S((i,j)_n)=(i^{'},j^{'})_n$, where $((i^{'},j^{'})_n)= (i,j)_{n+1}$. We define the following map
$$\rho:\Omega\times \{2,3,\cdots, n\}^{\mathbb{N}}\to \Lambda$$
by $$\rho((\omega, (n_1,n_2,\cdots)))=((\omega_1,n_1), (\omega_2,n_2), (\omega_3,n_3),\cdots).$$
Evidently,  $\rho$ is a bijection and $\rho \circ (\sigma\times \sigma^{'})=S\circ \rho$. This leads to the following theorem.

\begin{Theorem}\label{Importantresult}
The induced transformation $I$ is intrinsically ergodic with maximal maximal entropy $\log (2n-2)$.
\end{Theorem}

\begin{proof} Let $m$ be the product $\left(\dfrac{1}{2n-2},\dfrac{1}{2n-2}, \cdots, \dfrac{1}{2n-2} \right)$ measure on $\Lambda$. Note that $m$ is shift invariant, and is intrinsically ergodic. Since $\rho$ is a commuting bijection, the measure $m\circ \rho$ is $\sigma\times \sigma^{\prime}$-invariant and is intrinsically ergodic. By Lemma \ref{isomorphism},
$m\circ \rho\circ \phi$ is the unique measure of maximal entropy for $I$. Since entropy is preserved under an isomorphism,  the maximal entropy is $\log (2n-2)$.
\end{proof}

\section{Invariant measures for $K$}

It is a classical fact that if $\nu$ is an $I$-invaraint probability measure on $\Omega\times [a,b]$, then 
the probability measure $\mu$ defined on $\Omega\times [T_1(a),T_0(b)]$ by
$$\mu(E)=\frac{1}{\int \tau\,d\nu} \sum_{n\ge 0}\nu(\{(\omega,x)\in \Omega\times [a,b]:\tau(\omega,x)>n\}\cap K^{-n}(E))$$
is a $K$-invariant probability measure. So for any measure $\nu$ as defined in the provious section corresponds a $K$-invariant measure.

Now we consider the intrinsically ergodic measure of $K$. It can be found via Parry's work, see \cite{Walters1}.
For the sake of convenience, we give a brief introduction to Parry's result.
Given any one-dimensional subshift of finite type with irreducibility condition, the Parry measure given by a probability vector $(p_0, p_1,\cdots, p_{k-1})$ and stochastic matrix $(p_{ij})$  is constructed as follows. If $\lambda$ is the largest positive eigenvalue of $A$ ($A=(a_{ij})$ is the adjacency matrix of the subshift of finite type) and  $(u_0, u_1,\cdots, u_{k-1})$ is a strictly positive left eigenvector and $(v_0, v_1,\cdots, v_{k-1})$ is a strictly positive right eigenvector with $\sum_{i=0}^{k-1}u_iv_i=1$, then $p_i=u_iv_i$ and $p_{ij}=\dfrac{a_{ij} v_j}{\lambda v_i}$. We state the following classical result.

\begin{Theorem}\label{Parrymeasure}
Given any one-dimensional subshift of finite type with irreducibility condition, then the Parry measure is the intrinsically ergodic measure for this subshift of finite type. The maximal  entropy is $\log \lambda$.
\end{Theorem}

Recall the definition of $\beta$, 
given $n\geq 3$, let $\beta$ be the largest positive root of the following equation:
$$\beta^n=\sum_{i=0}^{n-2}\beta^i.$$
We can partition  $[T_1(a), T_0(b)]$ in terms of the image of $[a, b]$ under $K$. More precisely, let
$$\{[T_0^kT_1(a), T_0^{k+1}T_1(a)], 0\leq k\leq n-2, [a,b], [T_1^iT_0(b), T_1^{i+1}T_0(b)], 1\leq i\leq n-1\}$$
be a Markov partition of $[T_1(a), T_0(b)]$, where $T^0_{j}=id, j=0,1$. It is easy to see that  the image of each set of the Markov partition is the union of some sets of this partition. For instance, when $n=3$, let 
$$A=[T_1(a), T_0T_1(a)], B=[ T_0T_1(a), a],  C=[a,b],D=[ b, T_1T_0(b)],E=[  T_1T_0(b), T_0(b)].$$
Evidently, 
$$T_0(A)=B, T_0(B)=C, T_0(C)=D\cup E, T_1(C)=A\cup B, T_1(D)=C, T_1(E)=D.$$
Hence the associated adjacency matrix for this Markov partition is 
$$
 S_3=\begin{pmatrix}
  0 & 1 & 0 & 0& 0 \\
   0 & 0& 1 & 0& 0 \\ 
 1 & 1 & 0 & 1& 1 \\ 
 0 & 0& 1 & 0& 0 \\
 0 & 0 & 0 & 1& 0 \\
 \end{pmatrix}
.$$
This matrix can generate a subshift of finite type,  denoted by $\Sigma_3$, i.e. $$ \Sigma_3=\{({i_n})\in\{1,2,3,4,5\}^{\mathbb{N}}:S_{3_{{i_n}, {i_{n+1}}}}=1\}.$$
Similarly, for general $n$, we can find the adjacency  matrix $S_n$ and its corresponding subshift of finite type $\Sigma_n$. 
It is easy to see that the matrix $S_n$ is irreducible. Hence, we can make use of Parry's idea to find the unique measure of maximal entropy.

Denote $a_n=\det{(\lambda E-S_n)}$. The following lemma is doing some trivial calculation in linear algebra. 
\setcounter{Lemma}{1}
\begin{Lemma}
$a_{n+1}=\lambda^2 a_n-2\lambda^n$ for any $n\geq 3$, and $a_3=\lambda^2(\lambda^3-2\lambda-2)$. 
By induction, we have 
$$a_n=\lambda^{n-1}(\lambda^n-2(1+\lambda+\lambda^2+\cdots+\lambda^{n-2})).$$
The right eigenvector of $S_n$ is
$$\vec{v}=(v_0,v_1,\cdots, v_{2n-2})=(c, \lambda c, \lambda^2 c, \cdots, \lambda^{n-2} c, \lambda^{n-1} c, \lambda^{n-2}c, \lambda^{n-3}c, \cdots, \lambda c, c)$$
where $c> 0$. 

The left eigenvector  of $S_n$, denoted by $\vec{u}=(u_0, u_1, u_2,\cdots, u_{2n-2})$, is 
\begin{eqnarray*}
\left(d, \dfrac{1+\lambda}{\lambda}d, \dfrac{1+\lambda+\lambda^2}{\lambda^2}d,\cdots, \dfrac{1+\lambda+\cdots+\lambda^{n-2}}{\lambda^{n-2}}d,\lambda d,\dfrac{1+\lambda+\cdots+\lambda^{n-2}}{\lambda^{n-2}}d,\cdots, \dfrac{1+\lambda}{\lambda}d,d \right)
\end{eqnarray*}
where $d> 0$. 
By the construction of the Parry measure, we assume  $\vec{u}\cdot \vec{v}=1$, which implies that  $c$ and $d$ have following relation
$$\dfrac{1}{cd}=\dfrac{2}{\lambda-1}\left(\lambda^{n-1}-n+\dfrac{\lambda^n}{2}\right)+\lambda^n.$$
\end{Lemma}
Now we can find the Parry measure as follows, given any $(a_1a_2\cdots a_k)\in\{1,\cdots 2n-2\}^{k}$, the Parry measure defined on the cylinder 
$[a_1a_2\cdots a_k]$ is 
$$\mu([a_1a_2\cdots a_k])=p_{a_1}p_{a_1a_2}\cdots p_{a_{k-1}a_k}.$$

Let $\nu$ be the induced measure of $\mu$ on $\Omega\times [a,b]$, that is 
$$\nu(E)=\dfrac{\mu(E)}{\mu(\Omega\times [a,b])},$$
for $E$ a measurable subset of $\Omega\cap [a,b]$. By Abramov formula, 
$$h(K, \mu)=h(I, \nu)\times \mu(\Omega\times [a,b]),$$
where $h$ denotes the entropy of the underlying system, and $I=K_{\Omega\times [a,b]}$.
By the construction of the Parry measure, 
$$h(I, \nu)=\dfrac{\log \lambda}{u_n v_n}=\dfrac{\log \lambda}{cd\lambda^n}.$$

To prove the remaining part of Theorem \ref{Main11111}, we need to compare $h(I, \nu)=\dfrac{\log \lambda}{cd\lambda^n}$ with $\log(2n-2)$, the maximal entropy of $I$.

\begin{Lemma}\label{Compare}
For any $n\geq 3$, 
 $$\log(2n-2)>\dfrac{\log \lambda}{cd\lambda^n}.$$
\end{Lemma}
\begin{proof}
For $n=3$, we have to show $$\log6>\dfrac{\log \lambda}{cd\lambda^3}=\dfrac{\log \lambda}{\lambda^3}\left(\dfrac{2}{\lambda-1}\left(\lambda^2-3+\dfrac{\lambda^2}{2}\right)\right).$$
This is trivial as we can find the exact value of $\lambda$ in terms of some polynomial. Similarly, for $n=4$ the lemma is  still correct. 
Hence, it suffices to prove this lemma when $n\geq 5.$
Note that $\lambda $ is the largest  positive root of the following equation
$$\lambda^n-2(1+\lambda+\lambda^2+\cdots+\lambda^{n-2})=0.$$
Since $S_n$ is irreducible, it follows by Perron-Frobenius Theorem that such a $\lambda$ exists, and furthermore $1<\lambda<2$. By the construction of the Parry measure,  it follows that $$\dfrac{1}{cd}=\dfrac{2}{\lambda-1}\left(\lambda^{n-1}-n+\dfrac{\lambda^n}{2}\right)+\lambda^n.$$
Hence,  in order to prove 
$$\dfrac{\log \lambda}{cd\lambda^n}=\left[\dfrac{2}{\lambda-1}\left(\lambda^{n-1}-n+\dfrac{\lambda^n}{2}\right)+\lambda^n\right]\dfrac{\log \lambda}{\lambda^n}<\log(2n-2),$$ it suffices to prove that 
$$\lambda^{\dfrac{2}{(\lambda-1)\lambda^n}\left(\lambda^{n-1}-n+\frac{\lambda^n}{2}\right)}< \dfrac{2n-2}{\lambda}.$$
Since $n\geq 5$ and $1<\lambda<2$,  it follows that $\dfrac{2n-2}{\lambda}\geq \dfrac{8}{\lambda}\geq \lambda^2$. 
Hence it remains to show that 
$$\dfrac{2}{(\lambda-1)\lambda^n}\left(\lambda^{n-1}-n+\dfrac{\lambda^n}{2}\right)<2.$$
However, this inequality  immediately follows  from 
$$\lambda^n-2(1+\lambda+\lambda^2+\cdots+\lambda^{n-2})=0$$ 
and $1<\lambda<2$. 
\end{proof}
\begin{proof}[Proof of Theorem \ref{Main11111}]
By Lemma \ref{Compare}, Theorem \ref{Importantresult} and Theorem \ref{Parrymeasure}, we finish the proof of Theorem \ref{Main11111}.
\end{proof}

\section{Some remarks}
The shrinking random $\beta$-transformation we defined is very special. For a general sub switch region, i.e. $(a,b)\subset [\beta^{-1}, \beta^{-1}(\beta-1)^{-1}]$, does the intrinsically ergodic measure exist?  For general $1<\beta<2^{-1}(1+\sqrt{5})$, how can we find an invariant measure (or intrinsically ergodic measure) for the  shrinking random $\beta$-transformation? In the setting of classical random beta transformation, similar questions can be considered, see \cite{SD}. 
\section*{Acknowledgements}
 The second  author was supported the National Natural Science Foundation of China no. 11271137 and 
 by China Scholarship Council grant number 201206140003.

\label{}





\end{document}